\documentclass[a4paper,10pt]{amsart}
\usepackage{pb-diagram}
\usepackage{xypic}
\usepackage{amscd}
\usepackage{bm}
\usepackage{amsmath,amssymb,amscd,bbm,amsthm,mathrsfs,hyperref}
\usepackage{latexsym}
\usepackage{amsfonts}
\usepackage{amsthm}
\usepackage{indentfirst}
 \newtheorem{thm}{Theorem}[section]
 \newtheorem{cor}[thm]{Corollary}
 \newtheorem{lem}[thm]{Lemma}
 \newtheorem{prop}[thm]{Proposition}

 \newtheorem{rem}[thm]{Remark}

 \def\k{\mathbbm{k}}

 \newcommand{\Hom}{\mathrm{Hom}}

\title{Some dimensions of DG polynomial algebras}
\author{X.-F. Mao}
\address{Department of Mathematics, Shanghai University, Shanghai 200444, China}
\email{xuefengmao@shu.edu.cn}

\author{M.-Y.Zhang}
\address{Department of Mathematics, Shanghai University, Shanghai 200444, China}
\email{zmy1023@shu.edu.cn}
\date{}

\begin{document}

\maketitle \def\abstactname{abstact}
\begin{abstract}
 Assume that $\mathcal{A}$ is a cochain DG polynomial algebra such that its underlying graded algebra $\mathcal{A}^{\#}$ is a polynomial algebra generated by $n$ degree $1$ elements. We determine the DG Krull dimension, the global dimension, the ghost dimension and the Rouquier dimension of $\mathcal{A}$.

\end{abstract}

\maketitle

\section{introduction}
Let $k$ is an algebraically closed field of characteristic $0$. This paper deals with DG polynomial $k$-algebras, which are connected cochain DG algebras whose underlying graded algebras are polynomial algebras generated by degree one elements. In \cite{MGYC}, the differential structures and homological properties of  DG polynomial algebras are systematically studied. However, we still know very little about various homological invariants of DG polynomial algebras.

In dimension theory for rings, it is well known that that the Krull dimension and the global dimension of the polynomial ring $k[x_1,\cdots, x_n]$ are both equal to $n$.
The DG Krull dimension and the global dimension for DG algebras have been introduced in \cite{BSW} and \cite{MW3}, respectively. It is natural for one to be intensely curious about the DG Krull dimension and the global dimension of DG polynomial rings.  In group theory, the terminology `class'
 is used to measure the shortest length of a
filtration with sub-quotients of certain type.
Avramov-Buchweitz-Iyengar  \cite{ABI} introduced free class,
projective class and flat class for differential modules over a
commutative ring.  Inspired from their work, the first author and Wu \cite{MW3} introduced the notion of `DG free class' for semi-free DG modules over a DG algebra $\mathcal{A}$.
In brief, the DG free class of a semi-free DG $\mathcal{A}$-module $F$ is the shortest length of all strictly
increasing semi-free filtrations. For a more general DG module $M$, the first author and Wu \cite{MW3} introduced the invariant `cone length', which is the minimal DG free class of all its semi-free resolutions. The studies of these invariants for DG modules can be traced back to Carlsson's work in 1980s.
In \cite{Car}, Carlsson studied `free class' of solvable free DG
modules over a graded polynomial ring $R$ in $n$ variables of positive degree. By \cite[Theorem 16]{Car}, one sees that the cone length of any
totally finite DG $R$-module $M$ satisfies the inequality $\mathrm{cl}_RM\le n$. Note that the invariant `$\mathrm{cl}_RM$' was denoted by `$l(M)$' in \cite{Car}. The invariant `cone length' of a DG $A$-module plays a similar role in DG
homological algebra as the `projective dimension' of a module over a ring
does in classic homological ring theory (cf.\cite{MW3}).
The left
(resp. right) global dimension of a connected DG algebra $\mathcal{A}$
is defined to be the supremum of the set of cone lengths of all DG
$\mathcal{A}$-modules (resp. $\mathcal{A}\!^{op}$-modules).
The difficulty in the studies of the DG Krull dimension and the global dimension of DG polynomial rings comes from the fact that these two invariants are determined by a combination of the graded algebra structure and the differential system.
Due to the classifications of DG polynomial algebras in \cite{MGYC}, we show the following theorem (see Theorem \ref{Gldim} and Theorem \ref{dgkrull}):
\\
\begin{bfseries}
Theorem \,A.
\end{bfseries}
Assume that $\mathcal{A}$ is a cochain DG polynomial algebra such that $\mathcal{A}^{\#}=k[x_1,\cdots,x_n]$ with $|x_i|=1, i=1,2,\cdots, n$.
Then $\mathrm{DGdim}\mathcal{A}=n$ and
\begin{align*}
l.\mathrm{Gl.dim}\mathcal{A}=r.\mathrm{Gl.dim}\mathcal{A}=\begin{cases}
n, \,\,\text{if}\,\, \partial_{\mathcal{A}}=0\\
\frac{n(n-1)}{2}, \,\,\text{if}\,\, \partial_{\mathcal{A}}\neq 0.
\end{cases}
\end{align*}

 Christensen \cite{Chr} introduces the invariant `ghost length' for objects in triangulated categories. This concept can be applied to complexes and DG modules.  The ghost length together with other invariants, including level, cone length and trivial
category, of DG modules are studied in \cite{Kur1,Kur2, Mao}.
In \cite{HL1,HL2}, Hovey-Lockridge introduce the ghost dimension of a ring, which is defined as the maximum ghost lengthes of all perfect complexes.
One can naturally extend this definition to DG algebras.
Another important invariant of DG modules is called level. In the derived category $\mathscr{D}(\mathcal{A})$ of a DG algebra $\mathcal{A}$,
the level of a DG $\mathcal{A}$-module $M$ counts
the number of steps required to build $M$ out of $\mathcal{A}$ via triangles.  In \cite{ABIM}, some important and fundamental
properties of the level of DG modules are investigated.  To study topological spaces with
categorical representation theory, Kuriayashi \cite{Kur1} introduced the
levels for space and gave a
general method for computing the level of a space and studied the
relationship between the level and other topological invariants such
as Lusternik-Schnirelmann category. Later,
Schmidt \cite{Sch} used the properties of level of DG modules to study the
structure of the Auslander-Reiten quiver of some important cochain
DG algebra. In this paper, we show that $\mathrm{gh.len}_{\mathcal{A}}M+1 =
  \mathrm{level}_{\mathscr{D}(\mathcal{A})}^{\mathcal{A}}(M)$, for any compact DG $\mathcal{A}$-module $M$.
This implies that $\mathrm{Rouq.dim}\,\mathcal{A}=l.\mathrm{gh.dim}(\mathcal{A})+1$. Especially, we prove the following theorem for DG polynomial algebras (see Theorem \ref{ghrodim}).
\\
\begin{bfseries}
Theorem \, B.
\end{bfseries}
Assume that $\mathcal{A}$ is a cochain DG polynomial algebra such that $\mathcal{A}^{\#}=k[x_1,\cdots,x_n]$ with $|x_i|=1, i=1,\cdots, n$.
 Then
$\mathrm{gh.len}_{\mathcal{A}}k=\mathrm{cl}_{\mathcal{A}}k =\mathrm{level}_{\mathscr{D}(\mathcal{A})}^{\mathcal{A}}(k)-1$  and $$l.\mathrm{gh.dim}(\mathcal{A})+1=\mathrm{Rouq.dim}\,\mathcal{A}=\begin{cases}
n+1,\quad\text{if}\,\,\,\, \partial_{\mathcal{A}}=0,\\
\frac{n(n-1)}{2}+1, \quad \text{if}\,\,\, \partial_{\mathcal{A}}\neq 0.
 \end{cases} $$

\section{Notations and some invariants of dg algebras}
 We assume that the reader is familiar with basic definitions concerning DG homological
algebra. If this is not the case, we refer to \cite{FHT, MW1,MW2} for more details on them.

\subsection{Notations and terminology} Recall that a \emph{cochain} DG algebra is
a graded
$k$-algebra $\mathcal{A}$ together with a differential $\partial_{\mathcal{A}}: \mathcal{A}\to \mathcal{A}$  of
degree $1$ such that
\begin{align*}
\partial_{\mathcal{A}}(ab) = (\partial_{\mathcal{A}} a)b + (-1)^{|a|}a(\partial_{\mathcal{A}} b).
\end{align*}
And a cochain DG algebra $\mathcal{A}$ is called \emph{connected} if $\mathcal{A}^0=k$.
In the section, $\mathcal{A}$ will always be a connected cochain DG algebra
over a field $k$, if no special assumption is emphasized.
 We use
 $$H(\mathcal{A})=\bigoplus_{i=0}^{+\infty}\frac{\mathrm{ker}(\partial_{\mathcal{A}}^i)}{\mathrm{im}(\partial_{\mathcal{A}}^{i-1})}$$ to denote its
 cohomology graded algebra.
 Given a cocycle element $z\in \mathrm{ker}(\partial_{\mathcal{A}}^i)$, we write $\lceil z \rceil$ as the cohomology class in $H(\mathcal{A})$ represented by $z$. We denote $\mathcal{A}\!^{op}$ as the \emph{opposite} DG algebra of $\mathcal{A}$, whose
product is defined as $a \cdot b = (-1)^{|a|\cdot|b|}ba$ for all
graded elements $a$ and $b$. Right DG modules over $\mathcal{A}$ can be
identified with DG $\mathcal{A}\!^{op}$-modules. We write $\frak{m}$ as the maximal DG ideal $\mathcal{A}^{>0}$ of $\mathcal{A}$.
Via the canonical surjection $\varepsilon: \mathcal{A}\to k$, $k$ is both a DG
$\mathcal{A}$-module and a DG $\mathcal{A}\!^{op}$-module. The underlying graded algebra of $\mathcal{A}$ is written by $\mathcal{A}^{\#}$.

 Let $M$ be a DG $A$-module. We write $M^{\#}$ as its underlying graded $A^{\#}$-module, which is obtained by forgetting its differential.
 For any $i\in \Bbb{Z}$, the $i$-th \emph{suspension} of a DG
$\mathcal{A}$-module $M$ is the DG $\mathcal{A}$-module $\Sigma^i M$ defined by
$(\Sigma^iM)^j = M^{j+i}$. If $m \in M^l,$ the corresponding element
in $(\Sigma^i M)^{l-i}$ is denoted by $\Sigma^i m$. The action of
$\mathcal{A}$ on $\Sigma^iM$ is
$a\,(\Sigma^im) = (-1)^{|a|i}\Sigma^i(am),$
for all graded elements $ a\in \mathcal{A}$ and  $\Sigma^im\in \Sigma^iM$.
The differential $\partial_{\Sigma^iM}$ of $\Sigma^iM$ is defined by
$\partial_{\Sigma^iM}(\Sigma^im) = (-1)^i\Sigma^i\partial_{M}(m),$
for all graded elements $m\in M$. We set $\Sigma M = \Sigma^1 M$.

The category of DG $A$-modules is denoted by $\mathscr{C}(\mathcal{A})$ whose
morphisms are DG morphisms. The homotopy category $\mathscr{K}(\mathcal{A})$ is the quotient category of
$\mathscr{C}(\mathcal{A})$, whose objects are the same as those of
$\mathscr{C}(\mathcal{A})$ and whose morphisms are the homotopic equivalence
classes of morphisms in $\mathscr{C}(\mathcal{A})$.
The derived category of DG $\mathcal{A}$-modules is denoted by $\mathscr{D}(\mathcal{A})$, which is constructed from the category
$\mathscr{C}(\mathcal{A})$ by inverting quasi-isomorphisms
(\cite{We},\cite{KM}).

 Let $M$ be a DG $\mathrm{A}$-module.
We say $M$ is \emph{locally finite} if each $H^i(M)$ is a finite-dimensional $k$-vector space. If $H(M)$ is finite-dimensional as a $k$-vector space, then we say $M$ is \emph{totally finite}  (cf.\cite[Definition 12]{Car}). One sees that $M$ is totally finite if and only if it is both locally finite and homologically bounded. So it makes sense to denote
the full subcategory of $\mathscr{D}(\mathcal{A})$ consisting of totally finite DG $\mathcal{A}$-modules by $\mathscr{D}^b_{lf}(\mathcal{A})$. A DG $\mathcal{A}$-module  $M$ is  called \emph{compact} if the functor $\Hom_{\mathscr{D}(A)}(M,-)$ preserves
all coproducts in $\mathscr{D}(\mathcal{A})$.
 By \cite[Proposition 3.3]{MW1},
a DG $\mathcal{A}$-module  is compact if and only if it admits a minimal semi-free resolution with a finite semi-basis. The full subcategory of $\mathscr{D}(\mathcal{A})$ consisting of compact DG $\mathcal{A}$-modules is denoted by $\mathscr{D}^c(\mathcal{A})$.

\subsection{DG Krull dimension}In classical ring theory, the Krull dimension of a ring is one of the most important invariants.
In \cite{BSW}, this invariant was generalized to DG rings. Recall that a DG ideal $I\subseteq \mathcal{A}$ is called \emph{prime} if $I^{\#}$ is a graded prime ideal of $\mathcal{A}^{\#}$. Let $\mathrm{DGSpec}(\mathcal{A})$ denote the set of DG prime ideals of $\mathcal{A}$.
  According to \cite[Definition 2.5]{BSW}, the \emph{DG Krull dimension} of $\mathcal{A}$, denoted $\mathrm{DGdim}\,\mathcal{A}$, is the supremum of lengths of chains of DG prime ideals of $\mathcal{A}$.

\subsection{Global dimension}
Let $F$ be a semi-free DG $\mathcal{A}$-module. The \emph{DG free class} of $F$ is
defined to be the number
$$ \inf\{n\in \Bbb{N}\cup \{0\}\,|\, F \,\text{admits a strictly
increasing semi-free filtration of length}\,\, n\}.$$ We denote it
as $\mathrm{DG free\,\, class}_{\mathcal{A}}F$. Let $M$ be a non-quasi-trivial DG $\mathcal{A}$-module. The \emph{cone length} of $M$ is defined to be
the number
$$\mathrm{cl}_{\mathcal{A}}M =
\inf\{\mathrm{\,DGfree\,\,class}_{\mathcal{A}}F\,|\,F \stackrel{\simeq}\to M
 \ \text{is a semi-free resolution of}\  M\}.$$ And we define $\mathrm{cl}_{\mathcal{A}} N=-1$ if $H(N)=0$.
These two invariants of DG modules were introduced and studied in \cite{MW3}.
One sees that $\mathrm{cl}_{\mathcal{A}}M$ is just the invariant `$l(M)$' defined in \cite[Definition 9]{Car}
when $\mathcal{A}$ is a DG polynomial algebra with zero differential.
 According to \cite[Definition 5.1]{MW3}, the \emph{left global
dimension} and the \emph{right global dimension} of $\mathcal{A}$ are respectively
defined by
$$l.\mathrm{Gl.dim}\mathcal{A} = \sup\{\mathrm{cl}_{\mathcal{A}}M| M\in \mathscr{D}(\mathcal{A})\}
\,\, \text{and}\,\, r.\mathrm{Gl.dim}\mathcal{A} =
\sup\{\mathrm{cl}_{\mathcal{A}\!^{op}}M| M\in \mathscr{D}(\mathcal{A}\!^{op})\}.$$

\begin{rem}\label{clsmaller}
Note that $\mathrm{cl}_{\mathcal{A}}M$ may be $+\infty$. By the existence of
 Eilenberg-Moore resolution, we have
 $\mathrm{cl}_{\mathcal{A}}M \le \mathrm{pd}_{H(\mathcal{A})}H(M)$ and hence $l.\mathrm{Gl.dim}\mathcal{A}\le \mathrm{gl.dim}H(\mathcal{A})$.
\end{rem}

\subsection{Ghost dimension}
A DG morphism $f: M\to N$ in $\mathscr{D}(\mathcal{A})$ is called a \emph{ghost
morphism} if $H(f)=0$. The concept of `ghost length' was first introduced by Christensen in \cite{Chr}.
Later, Hovey-Lockridge \cite{HL1} and  Kuribayashi
\cite{Kur1}
 applied this invariant to complexes and DG modules, respectively.
 A DG $\mathcal{A}$-module $M$ is said to have \emph{ghost
length} $n$, written by $\mathrm{gh.len}_{\mathcal{A}} M = n$, if every composite
$$M\stackrel{f_1}{\to} I_1\stackrel{f_2}{\to} \cdots
\stackrel{f_{n+1}}{\to} I_{n+1}$$ of $n+1$ ghosts is $0$ in
$\mathscr{D}(\mathcal{A})$, and there is a composite of $n$ ghosts from $M$
that is not $0$ in $\mathscr{D}(\mathcal{A})$.
We set $\mathrm{gh.len}_{\mathcal{A}}M=-1$
if $M$ is zero object in $\mathscr{D}(\mathcal{A})$.  In \cite{HL1,HL2}, Hovey-Lockridge introduced and studied ghost dimension for rings. In a similar way,  we define the left ghost dimension of $\mathcal{A}$ as
$$l.\mathrm{gh.dim}(\mathcal{A}) = \sup\{\mathrm{gh.len}_{\mathcal{A}}M | M\in
\mathscr{D}^c(\mathcal{A})\}.$$
Similarly, we can define the right global dimension of $\mathcal{A}$ as $$r.\mathrm{gh.dim.}A =
\sup\{\mathrm{gh.len}_{\mathcal{A}^{op}}M | M\in \mathscr{D}^c(\mathcal{A}^{op})\}.$$
\subsection{Rouquier dimension}\label{rouqdim}
Let $\mathcal{C}$ be a subcategory or simply a set of some objects
of $\mathscr{D}(\mathcal{A})$. We denote by $\mathrm{smd}(\mathcal{C})$ the
minimal strictly full subcategory which contains $\mathcal{C}$ and
is closed under taking (possible) direct summands. And we write
$\overline{\mathrm{add}}(\mathcal{C})$ (resp.
$\mathrm{add}(\mathcal{C})$) as the intersection of all strict and
full subcategories of $\mathscr{D}(\mathcal{A})$ that contain $\mathcal{C}$ and
are closed under direct sums (resp. finite direct sums) and all
suspensions.

Let $\mathcal{S}$ and $\mathcal{T}$ be two strict and full
subcategories of $\mathscr{D}(\mathcal{A})$. We define $\mathcal{S}\star
\mathcal{T}$ as a full subcategory of $\mathscr{D}(\mathcal{A})$,  whose
objects are described as follows: $M\in \mathcal{S}\star
\mathcal{T}$ if and only if there is an exact triangle
$$ L\to M\to N\to \Sigma L, $$
where $L\in \mathcal{S}$ and $N\in \mathcal{T}$. For any strict and
full subcategory $\mathcal{R}$ of $\mathscr{D}(\mathcal{A})$, one has
$\mathcal{R}\star (\mathcal{S}\star \mathcal{T}) = (\mathcal{R}\star
\mathcal{S})\star \mathcal{T}$ (see \cite{BBD} or \cite[1.3.10]{BVDB}). Thus, the
following notation is unambiguous:
\begin{equation*}
\mathcal{T}^{\star n} =
\begin{cases} 0\quad&\text{for}\, n =0;\\
\mathcal{T}\quad & \text{for}\, n =1; \\
\overbrace{\mathcal{T}\star\cdots\star
\mathcal{T}}^{n\,\text{copies}} \quad&\text{for}\, n\ge 1.
\end{cases}
\end{equation*}
We refer to the objects of $\mathcal{T}^{\star n}$ as $(n-1)$-fold
extensions of objects from $\mathcal{T}$. Define $\mathcal{S}\diamond\mathcal{T} =
\mathrm{smd}(\mathcal{S}\star\mathcal{T})$. Inductively, we define
$$\langle\mathcal{S}\rangle_1 =
\mathrm{smd}(\mathrm{add}(\mathcal{S})), \text{and}\,
\langle\mathcal{S}\rangle_n =
\langle\mathcal{S}\rangle_{n-1}\diamond \langle\mathcal{S}\rangle_1,
n\ge 2.$$ Similarly,  we define
$\overline{\langle\mathcal{S}\rangle}_1 =
\mathrm{smd}(\overline{\mathrm{add}}(\mathcal{S})),
\overline{\langle\mathcal{S}\rangle}_n =
\overline{\langle\mathcal{S}\rangle}_{n-1}\diamond
\overline{\langle\mathcal{S}\rangle}_1, n\ge 2$. We have the
associativity of $\diamond$ and the formula
$$\mathcal{C}_1\diamond\mathcal{C}_2\diamond\cdots\diamond
\mathcal{C}_n
=\mathrm{smd}(\mathcal{C}_1\star\cdots\star\mathcal{C}_n)$$ (see
\cite[Section 2]{BVDB}). Denote $\langle\mathcal{S}\rangle=
\bigcup_{i\ge 0}\langle\mathcal{S}\rangle_i.$
For an DG $\mathcal{A}$-module $M$, its \emph{ $\mathcal{A}$-level} is defined to be
$$\mathrm{level}_{\mathscr{D}(\mathcal{A})}^{\mathcal{A}}(M)=\inf\{n\in \Bbb{N}\cup \{0\}| M\in
\langle \mathcal{A}\rangle_n\}.$$  This invariant is
originally introduced by Avramov, Buchweitz, Iyengar and Miller in
\cite{ABIM}. The $\mathcal{A}$-level of $M$ counts the number of steps required to build
$M$ out of $\mathcal{A}$ via triangles in $\mathscr{D}(\mathcal{A})$. It is a very
important numerical invariant in the study of the derived category
of compact DG $\mathcal{A}$-modules.
According to the definition in \cite{Rou},
we define the \emph{Rouquier dimension} $\mathrm{Rouq.dim}\,\mathcal{A} $ of $\mathcal{A}$ as the smallest $n$ such that
$\langle \mathcal{A}\rangle_n=\mathscr{D}^c(\mathcal{A})$. Actually, we have
$$\mathrm{Rouq.dim}\,\mathcal{A}= \sup\{\mathrm{level}_{\mathscr{D}(\mathcal{A})}^{\mathcal{A}}(M)| M\in \mathscr{D}^c(A)\}.$$
\begin{rem}\label{ghlevel}
By \cite[Proposition 7.5]{Kur1}, we have $\mathrm{gh.len}_{\mathcal{A}}M+1\le
\mathrm{level}_{\mathscr{D}(\mathcal{A})}^{\mathcal{A}}(\mathcal{M}),$ for any $M\in \mathscr{D}(\mathcal{A})$. Hence $l.\mathrm{gh.dim}\,\mathcal{A} +1 \le \mathrm{Rouq.dim}\,\mathcal{A}.$
\end{rem}

\section{some facts on dg  polynomial algebras}
In this section, we list some useful facts on the structures and classifications of DG polynomial algebras.
\begin{prop}\cite[Theorem 3.1]{MGYC}\label{diffstr}
Let $(\mathcal{A},\partial_{\mathcal{A}})$ be a DG polynomial algebra such that $\mathcal{A}^{\#}=k[x_1,x_2,\cdots,x_n]$ with $|x_i|=1$, for any $i\in \{1,2,\cdots,n\}$.
Then there exist some $t_1,t_2,\cdots,t_n\in k$ such that $\partial_{\mathcal{A}}$ is defined  by
$$\partial_{\mathcal{A}}(x_i)=\sum\limits_{j=1}^nt_jx_ix_j=\sum\limits_{j=1}^{i-1}t_jx_jx_i+t_ix_i^2+\sum\limits_{j=i+1}^nt_jx_ix_j,$$ for any $i\in \{1,2,\cdots,n\}$.
\end{prop}
It is reasonable to write the DG polynomial algebra $\mathcal{A}$ in Proposition \ref{diffstr} by $\mathcal{A}(t_1,t_2,\cdots, t_n)$. The set of DG polynomial algebras in $n$ degree one variables is
$$\Omega(x_1,x_2,\cdots,x_n)=\{\mathcal{A}(t_1,t_2,\cdots,t_n)|t_i\in
k, i=1,2,\cdots, n\}\cong \Bbb{A}_{k}^n.$$
When it comes to the isomorphism classes of DG polynomial algebras, we have the following proposition.
\begin{prop}\cite[Corollary 4.2]{MGYC}\label{isomor}
In space $\Omega(x_1,x_2,\cdots, x_n)$, there are only two isomorphism
 classes $\mathcal{A}(0,0,\cdots,0)$ and $\mathcal{A}(1,0,\cdots, 0)$.
\end{prop}
In general, the cohomology algebra of a DG algebra contains much important information on its properties (see \cite{AT,DGI}).
Proposition \ref{isomor} indicates that the computations of the cohomology algebra of all non-trivial DG polynomial algebras in $\Omega(x_1,x_2,\cdots, x_n)$ can be reduced to the calculation of $H(\mathcal{A}(1,0,\cdots,0))$.
We have the following proposition (see \cite[Proposition 5.1, Theorem 5.2]{MGYC}.
\begin{prop}\label{cohomology}
The DG polynomial algebra  $\mathcal{A}(1,0,\cdots,0)$ has formal property and its
cohomology graded algebra is the polynomial algebra  $$k[\lceil x_2^2\rceil,\lceil x_2x_3\rceil, \cdots,\lceil x_2x_n\rceil, \lceil x_3^2\rceil, \cdots, \lceil x_3x_n\rceil, \cdots,  \lceil x_{n-1}^2\rceil, \lceil x_{n-1}x_n\rceil, \lceil x_n^2\rceil ].$$
\end{prop}
By Proposition \ref{isomor} and Proposition \ref{cohomology},  the following corollary is immediate.
\begin{cor}\label{impcor}
Let $\mathcal{A}$ be a DG polynomial algebra in $\Omega(x_1,x_2,\cdots, x_n)$. Then $H(\mathcal{A})$ is a polynomial algebra and we have
\begin{align*}
\mathrm{depth}_{H(\mathcal{A})}H(A)=\mathrm{gl.dim}H(\mathcal{A})=\begin{cases}\frac{n(n-1)}{2}, \,\, \text{if}\,\,\, \partial_{\mathcal{A}}\neq 0 \\
                                      n, \,\, \text{if}\,\,\, \partial_{\mathcal{A}}=0.
\end{cases}
\end{align*}
\end{cor}

\section{global dimension and dg krull dimension }
In this section, we will compute the global dimension, the ghost dimension and the Rouquier dimension of a DG polynomial algebra in $\Omega(x_1,x_2,\cdots, x_n)$.

\begin{lem}\label{depthclk}
Let $\mathcal{A}$ be a connected cochain DG algebra such that $\mathrm{cl}_{\mathcal{A}}k$
is finite. Then $\mathrm{depth}_{H(\mathcal{A})}H(\mathcal{A})\le \mathrm{cl}_{\mathcal{A}}k$.
Furthermore, if $\mathrm{depth}_{H(\mathcal{A})}H(\mathcal{A}) = \mathrm{cl}_{\mathcal{A}}k$, then
we have
$$\mathrm{gldim}H(\mathcal{A})=\mathrm{depth}_{H(\mathcal{A})}H(\mathcal{A}) = \mathrm{cl}_{\mathcal{A}}k.$$
\end{lem}
\begin{proof}
In \cite[Theorem 4.8]{MW3}, the author and Wu give a detailed proof for the same statement, when $\mathcal{A}$ is an Adams connected DG algebra. It is easy for one to check that the proof there also carries through for a general connected cochain DG algebra.
\end{proof}

\begin{thm}\label{Gldim}
Let $\mathcal{A}$ be a DG polynomial algebra in $\Omega(x_1,x_2,\cdots, x_n)$. Then \begin{align*}
l.\mathrm{Gl.dim}\mathcal{A}=\mathrm{cl}_{\mathcal{A}}k=\mathrm{gl.dim}H(\mathcal{A})=\begin{cases}
n, \,\,\,\text{if}\,\,\, \partial_{\mathcal{A}}=0\\
\frac{n(n-1)}{2}, \,\,\text{if}\,\,\, \partial_{\mathcal{A}}\neq 0.
\end{cases}
\end{align*}
\end{thm}
\begin{proof}
By Corollary \ref{impcor}, we have
\begin{align*}
\mathrm{depth}_{H(\mathcal{A})}H(\mathcal{A})=\mathrm{gl.dim}H(\mathcal{A})=\begin{cases}\frac{n(n-1)}{2}, \,\, \text{if}\,\,\, \partial_{\mathcal{A}}\neq 0 \\
                                      n, \,\, \text{if}\,\,\, \partial_{\mathcal{A}}=0.
\end{cases}
\end{align*}
By Remark \ref{clsmaller}, we have $\mathrm{cl}_{\mathcal{A}}k\le \mathrm{pd}_{H(\mathcal{A})}k=\mathrm{gl.dim}H(\mathcal{A})<\infty$. On the other hand, $\mathrm{cl}_{\mathcal{A}}k\ge \mathrm{depth}_{H(\mathcal{A})}H(\mathcal{A})$ by Lemma \ref{depthclk}. Then we get
$$\mathrm{depth}_{H(\mathcal{A})}H(\mathcal{A})=\mathrm{cl}_{\mathcal{A}}k=\mathrm{gl.dim}H(\mathcal{A}).$$ Therefore,
$$l.\mathrm{Gl.dim}\mathcal{A}=\sup\{\mathrm{cl}_{\mathcal{A}}M| M\in \mathscr{D}(\mathcal{A})\}\ge \mathrm{cl}_{\mathcal{A}}k=\mathrm{gl.dim}H(\mathcal{A}). $$
By Remark \ref{clsmaller} again, we have $l.\mathrm{Gl.dim}\mathcal{A}\le \mathrm{gl.dim}H(\mathcal{A})$. This implies that
$$l.\mathrm{Gl.dim}\mathcal{A}=\mathrm{cl}_{\mathcal{A}}k=\mathrm{gl.dim}H(\mathcal{A}).$$
\end{proof}
\begin{rem}
For the DG polynomial algebra $\mathcal{A}$ in Theorem \ref{Gldim}, one can similarly show
\begin{align*}
r.\mathrm{Gl.dim}\mathcal{A}=\mathrm{cl}_{\mathcal{A}^{op}}k=\mathrm{gl.dim}H(\mathcal{A})=\begin{cases}
n, \,\,\,\text{if}\,\,\, \partial_{\mathcal{A}}=0\\
\frac{n(n-1)}{2}, \,\,\text{if}\,\,\, \partial_{\mathcal{A}}\neq 0.
\end{cases}
\end{align*}
\end{rem}
\begin{thm}\label{dgkrull}
Let $\mathcal{A}$ be a DG polynomial algebra in $\Omega(x_1,x_2,\cdots,x_n)$. Then $\mathrm{DGdim}\mathcal{A}=n$, i.e., the DG Krull dimension of $\mathcal{A}$ is independent of the choice of its differential.
\end{thm}
\begin{proof}
If $\partial_{\mathcal{A}}=0$, then any DG prime ideal of $\mathcal{A}$ is just the graded prime ideal of the graded polynomial algebra $k[x_1,x_2,\cdots,x_n]$. Hence $$\mathrm{DGdim}\mathcal{A}=\mathrm{dim}k[x_1,x_2,\cdots,x_n]=n$$ in this case.  If $\partial_{\mathcal{A}}\neq 0$, then $\mathcal{A}\cong \mathcal{A}(1,0,\cdots,0)$ by Proposition \ref{isomor}. Therefore, $$\mathrm{DGdim}\mathcal{A}= \mathrm{DGdim}\mathcal{A}(1,0,\cdots,0)$$ and we only need to compute $\mathrm{DGdim}\mathcal{A}(1,0,\cdots,0)$.
 Visibly, we have $$\mathrm{DGdim}\mathcal{A}(1,0,\cdots,0)\le \dim \mathcal{A}(1,0,\cdots,0)^{\#}=\dim k[x_1,x_2,\cdots,x_n]=n$$ since
$P^{\#}\in \mathrm{Spec}\,\mathcal{A}(1,0,\cdots,0)^{\#}$ for any $P\in \mathrm{DGSpec}\,\mathcal{A}(1,0,\cdots,0)$. On the other hand,  we have $\partial_{\mathcal{A}(1,0,\cdots,0)}(x_i)=x_1x_i$ for any $i\in \{1,2,\cdots, n\}$. So each $(x_1,\cdots,x_i)$ is a DG prime ideal of $\mathcal{A}(1,0,\cdots,0)$ when $i\in \{1,2,\cdots, n\}$. Thus
$$(0)\subset (x_1)\subset (x_1,x_2)\subset \cdots \subset (x_1,\cdots,x_n)$$
is a chain of prime ideals of length $n$ and $\mathrm{DGdim}\mathcal{A}(1,0,\cdots,0)\ge n$.  Hence $\mathrm{DGdim}\mathcal{A}(1,0,\cdots,0)= n$.

\end{proof}

\section{some basic lemmas}
In this section, we assume that $\mathcal{A}$ is a connected cochain DG algebra. We list some fundamental lemmas, which will be used in the studies of ghost dimension and Rouquier dimension for DG polynomial algebras.

The following lemma can be proved by the so-called Ghost lemma. One can also see it in
\cite[Lemma 6.7]{Sch} and \cite[Proposition 7.5]{Kur1}. For a detailed proof, we refer the reader to \cite[Proposition 4.10]{Mao}.
\begin{lem}\label{ghlen}
Suppose that $M$ is a DG $\mathcal{A}$-module.  Then $M\in
\langle\overline{\mathcal{A}}\rangle_n$ if and only if $\mathrm{gh.len}_{\mathcal{A}}M\le n-1$.
\end{lem}

\begin{rem}\label{charghlen}
By Lemma \ref{ghlen}, one sees that $\mathrm{gh.len}_{\mathcal{A}}M= n-1$ if and only if $M\in \langle\overline{\mathcal{A}}\rangle_{n}$ and $M\not\in \langle\overline{\mathcal{A}}\rangle_{n-1}$, i.e., $\mathrm{gh.len}_{\mathcal{A}}M=\inf\{n-1|M\in \langle\overline{\mathcal{A}}\rangle_{n}\}$.
\end{rem}

\begin{lem}\label{charclen}
For any DG $\mathcal{A}$-module $M$, $\mathrm{cl}_{\mathcal{A}}M = n$ if and only if $M$
is an object in $\overline{\mathrm{add}(\mathcal{A})}^{\,\star n+1}$ but not
in $\overline{\mathrm{add}(\mathcal{A})}^{\,\star n}$.
\end{lem}
\begin{proof}
We only need to show that $M\in \overline{\mathrm{add}(\mathcal{A})}^{\,\star n+1}$ if and only if
$\mathrm{cl}_{\mathcal{A}}M \le n$.
If $\mathrm{cl}_{\mathcal{A}}M =l \le n$, then $M$ admits a semi-free resolution
$F$, which has a strictly increasing semi-free filtration
$$0=F(-1)\subset F(0)\subset\cdots\subset F(i)\subset
F(i+1)\subset \cdots \subset F(l) = F$$ of length $l$. This yields a
sequence of short exact sequences
$$0\to F(i-1)\to F(i) \to F(i)/F(i-1)\to 0,\,  1\le i\le l.$$
Since $F(0)$ and $F(i)/F(i-1), \,i\ge 1$ are in
$\overline{\mathrm{add}(\mathcal{A})}$, induction shows that $F$ is in
$\overline{\mathrm{add}(\mathcal{A})}^{\,\star l+1}$. Hence $M$ is an object
in $\overline{\mathrm{add}(\mathcal{A})}^{\,\star n+1}$.

Conversely,  let $M$ be an object in
$\overline{\mathrm{add}(\mathcal{A})}^{\,\star n+1}$ for some $n\ge 0$. We
will prove $\mathrm{cl}_{\mathcal{A}}M \le n$ by induction on $n$. For $n=0$,
the assertion is evident. For $n\ge 1$, there exists an exact
triangle
$$ L\stackrel{\varepsilon}{\to} N\to M \to \Sigma L,$$
where $L$ and $N$ are in $\overline{\mathrm{add}(\mathcal{A})}$ and
$\overline{\mathrm{add}(\mathcal{A})}^{\,\star n}$ respectively. Since $L$ is
in $\overline{\mathrm{add}(\mathcal{A})}$, the DG morphism $\varepsilon$ can
be represented by a DG morphism $\varepsilon': F_L \to F_N$, where
$F_L$ is a direct sum of shifted copies of $\mathcal{A}$ and $F_N$ is a
semi-free resolution of $N$ with a strictly increasing semi-free
filtration
$$0=F_N(-1)\subset F_N(0)\subset \cdots\subset
F_N(i)\subset \cdots \subset F_N(n)=F_N$$ of length $n$. Clearly,
$M\cong \mathrm{Cone}(\varepsilon')$ in $\mathrm{D}(\mathcal{A})$. We have the
following cone exact sequence
$$ 0\to F_N\stackrel{\phi}{\to} \mathrm{Cone}(\varepsilon')
\to \Sigma F_L\to 0.$$ Since $\phi$ is an injective DG morphism,
$$0\subset \phi(F_N(0)) \subset \cdots \subset \phi(F_N(i))\subset
\cdots \subset \phi(F_N(n))\subset \mathrm{Cone}(\varepsilon')$$ is
a strictly semi-free filtration of $\mathrm{Cone}(\varepsilon')$ of
length $n$. Since $\mathrm{Cone}(\varepsilon') \cong M$ in
$\mathrm{D}(\mathcal{A})$, we have $\mathrm{cl}_{\mathcal{A}}M \le n$.
\end{proof}
\begin{rem}\label{ghsmcl}
For any $n\in \Bbb{N}$, we have $\overline{\mathrm{add}(\mathcal{A})}^{\,\star n}\subset \langle\overline{\mathcal{A}}\rangle_{n}$.  So Lemma \ref{ghlen} and Lemma \ref{charclen} imply $\mathrm{gh.len}_{\mathcal{A}}M\le \mathrm{cl}_{\mathcal{A}}M $. Then
\begin{align*}
l.\mathrm{gh.dim}\,\mathcal{A} &= \sup\{\mathrm{gh.len}_{\mathcal{A}}M | M\in \mathscr{D}^c(\mathcal{A})\}                     \\
                               &\le \sup\{\mathrm{cl}_{\mathcal{A}}M  | M\in \mathscr{D}^c(\mathcal{A})\}                      \\
                               &\le \sup\{\mathrm{cl}_{\mathcal{A}}M  | M\in \mathscr{D}(\mathcal{A})\} = l.\mathrm{Gl.dim}\mathcal{A}.
\end{align*}

\end{rem}

\begin{lem}\label{rouqdim}
Let $\mathcal{A}$ be a connected cochain DG algebra such that $H(\mathcal{A})$ is a left Noetherian graded algebra. Then
$$\mathrm{Rouq.dim}\,\mathcal{A}\le \mathrm{gl.dim}H(\mathcal{A})+1.$$
\end{lem}
\begin{proof}
If $\mathrm{gl.dim}H(\mathcal{A})=+\infty$, then the inequality holds obviously. We only need to consider the case that $\mathrm{gl.dim}H(\mathcal{A})<+\infty$. For any $M\in \mathscr{D}^c(\mathcal{A})$, we have $\mathrm{pd}_{H(\mathcal{A})}H(M)\le \mathrm{gl.dim}H(\mathcal{A})<+\infty$. Set $d=\mathrm{pd}_{H(\mathcal{A})}H(M)$.   The left graded $H(\mathcal{A})$-module $H(M)$ is finitely generated since $M\in  \mathscr{D}^c(\mathcal{A})$.
Thus $H(M)$ admits a finitely generated minimal free resolution
\begin{align}\label{minfree}
\cdots 0\to H(\mathcal{A})\otimes
V_d\stackrel{\partial_d}{\to} \cdots \stackrel{\partial_2}{\to}
H(\mathcal{A})\otimes V_1 \stackrel{\partial_1}{\to} H(\mathcal{A})\otimes
V_0\stackrel{\varepsilon}{\to} H(M)\to 0,
\end{align}
where each $V_i$ is a finite-dimensional vector space.  By \cite[Proposition 20.11 ]{FHT}, the resolution (\ref{minfree}) induces a semi-free resolution $F$, which is called Eilenberg-Moore resolution, of $M$. From the proof of \cite[Proposition 20.11 ]{FHT}, one sees that $$F^{\#}=\bigoplus_{i=0}^d\mathcal{A}^{\#}\otimes \Sigma^iV_i,$$ and $F$ admits a semi-free filtration
$$0=F(-1)\subset F(0)\subset F(1)\subset \cdots \subset F(d-1)\subset F(d)=F, $$
where each $F(i)/F(i-1)=\mathcal{A}\otimes \Sigma^iV_i$, $i=1,2,\cdots, d$.
The semi-free filtration above yields a
sequence of short exact sequences
$$0\to F(i-1)\to F(i) \to F(i)/F(i-1)\to 0,\,  1\le i\le d.$$
Since  each $F(i)/F(i-1)=\mathcal{A}\otimes V_i \in \mathrm{add}(\mathcal{A})$, we prove $F\in \mathrm{add}(\mathcal{A})^{\,\star (d+1)}$ by induction.
Since $\mathrm{add}(\mathcal{A})^{\,\star (d+1)}\subset
\langle \mathcal{A}\rangle_{d+1}$, we show that $M\in
\langle \mathcal{A}\rangle_{d+1}$ and then $\mathscr{D}^c(\mathcal{A})\subset \langle \mathcal{A}\rangle_{d+1}$. Therefore, $\mathrm{Rouq.dim}\,\mathcal{A}\le \mathrm{gl.dim}H(\mathcal{A})+1$.

\end{proof}

\begin{lem}\label{restrict}\cite[Corollary 3.13]{Rou}
Let $\mathcal{E}$ be a subcategory or just a set of some objects
of $\mathscr{D}^c(\mathcal{A})$. Then $\mathscr{D}^c(\mathcal{A})\cap \overline{\langle
\mathcal{E} \rangle}_{i} = \langle \mathcal{E} \rangle_{i},i\ge 0$.
\end{lem}
\begin{lem}\label{ghlenlevel}
For any $M\in \mathscr{D}^c(\mathcal{A})$, we have $\mathrm{gh.len}_{\mathcal{A}}M=\mathrm{level}_{\mathscr{D}(\mathcal{A})}^{\mathcal{A}}(M)-1$. And hence
$l.\mathrm{gh.dim}\,\mathcal{A}+1=\mathrm{Rouq.dim}\,\mathcal{A}$.
\end{lem}
\begin{proof}
We have
\begin{align*}
\mathrm{gh.len}_{\mathcal{A}}M&\stackrel{(a)}{=}\inf\{n-1|M\in \langle\overline{\mathcal{A}}\rangle_{n}\}\\
                 &=\inf\{n-1|M\in  \mathscr{D}^c(\mathcal{A})\cap\langle\overline{\mathcal{A}}\rangle_{n}\} \\
                  &\stackrel{(b)}{=}\inf\{n-1|M\in \langle \mathcal{A} \rangle_{n} \} \\
                                                         &=\inf\{n|M\in \langle \mathcal{A} \rangle_{n} \} -1 \\
                                                         &=\mathrm{level}_{\mathscr{D}(\mathcal{A})}^{\mathcal{A}}(M)-1,
\end{align*}
where $(a)$ and $(b)$ are obtained by Remark \ref{charghlen} and Lemma \ref{restrict}, respectively. Then
\begin{align*}
\mathrm{Rouq.dim}\,\mathcal{A}&= \sup\{\mathrm{level}_{\mathscr{D}(\mathcal{A})}^{\mathcal{A}}(M)| M\in \mathscr{D}^c(\mathcal{A})\}             \\
                              &= \sup\{\mathrm{gh.len}_{\mathcal{A}}M+1| M\in \mathscr{D}^c(\mathcal{A})\}\\
                              &= \sup\{\mathrm{gh.len}_{\mathcal{A}}M| M\in \mathscr{D}^c(\mathcal{A})\}+1 \\
                              &= l.\mathrm{gh.dim}\,\mathcal{A}+1.
\end{align*}

\end{proof}

\section{ghost dimension and rouquier dimension}
In this section, we determine the ghost dimension and the Rouquier dimension of DG polynomial algebras. To achieve this goal, we need the following lemma.
\begin{lem}\cite[Theorem $5.1$]{ABIM}\label{keylem}
Let $\mathcal{A}$ be a DG algebra with zero differential and $U$ a DG $\mathcal{A}$-module. If the ring $\mathcal{A}^{\#}$ is a commutative, Noetherian algebra over a field,  then $\mathrm{level}_{\mathscr{D}(\mathcal{A})}^{\mathcal{A}}(U)\ge \mathrm{height}I+1$, where $I$ is the annihilator of the $\mathcal{A}^{\#}$-module $H(U)$.
\end{lem}

Applying Lemma \ref{keylem} and the listed lemmas in the previous section, we can prove the following theorem.
\begin{thm}\label{ghrodim}
Let $\mathcal{A}$ be a DG polynomial algebra in $\Omega(x_1,x_2,\cdots, x_n)$.  Then
$\mathrm{gh.len}_{\mathcal{A}}k=\mathrm{cl}_{\mathcal{A}}k =\mathrm{level}_{\mathscr{D}(\mathcal{A})}^{\mathcal{A}}(k)-1$  and $$l.\mathrm{gh.dim}(\mathcal{A})+1=\mathrm{Rouq.dim}\,\mathcal{A}=\begin{cases}
n+1,\quad\text{if}\,\,\,\, \partial_{\mathcal{A}}=0,\\
\frac{n(n-1)}{2}+1, \quad \text{if}\,\,\, \partial_{\mathcal{A}}\neq 0.
 \end{cases} $$
\end{thm}

\begin{proof}

By Theorem \ref{Gldim},
we have
\begin{align*}
l.\mathrm{Gl.dim}\mathcal{A}=\mathrm{cl}_{\mathcal{A}}\k=\mathrm{gl.dim}H(\mathcal{A})=\mathrm{pd}_{H(\mathcal{A})}k=\begin{cases}
n, \,\,\,\text{if}\,\,\, \partial_{\mathcal{A}}=0\\
\frac{n(n-1)}{2}, \,\,\text{if}\,\,\, \partial_{\mathcal{A}}\neq 0.
\end{cases}
\end{align*}
By Remark \ref{ghsmcl} and Lemma \ref{ghlenlevel}, we have $$\mathrm{gh.len}_{\mathcal{A}}k=\mathrm{level}_{\mathscr{D}(\mathcal{A})}^{\mathcal{A}}(k)-1\le \mathrm{cl}_{\mathcal{A}}k=\mathrm{pd}_{H(\mathcal{A})}k.$$

If $\partial_{\mathcal{A}}=0$, then $H(\mathcal{A})=\mathcal{A}$ and we have
$$\mathrm{level}_{\mathscr{D}(\mathcal{A})}^{\mathcal{A}}(k)=\mathrm{level}_{\mathscr{D}(H(\mathcal{A}))}^{H(\mathcal{A})}(k) \ge \mathrm{height}(\mathrm{Ann}_{H(\mathcal{A})}k)+1=n+1$$ by Lemma \ref{keylem}. Therefore, $\mathrm{gh.len}_{\mathcal{A}}k=\mathrm{level}_{\mathscr{D}(\mathcal{A})}^{\mathcal{A}}(k)-1=n$.

If $\partial_{\mathcal{A}}\neq 0$, then $\mathcal{A}\cong \mathcal{A}(1,0,\cdots, 0)$ by Proposition \ref{isomor}. By Proposition \ref{cohomology}, $\mathcal{A}(1,0,\cdots,0)$ has formal property and $H(\mathcal{A}(1,0,\cdots,0))$
 is the polynomial algebra  $$\k[\lceil x_2^2\rceil,\lceil x_2x_3\rceil, \cdots,\lceil x_2x_n\rceil, \lceil x_3^2\rceil, \cdots, \lceil x_3x_n\rceil, \cdots,  \lceil x_{n-1}^2\rceil, \lceil x_{n-1}x_n\rceil, \lceil x_n^2\rceil ].$$
Therefore, $\mathcal{A}$ is also formal and $$\mathrm{pd}_{H(\mathcal{A})}k=\mathrm{gl.dim}H(\mathcal{A})=\mathrm{height}(\mathrm{Ann}_{H(\mathcal{A})}k)=\frac{n(n-1)}{2}.$$
 By the definition of formality, $\mathcal{A}$ can be connected with the trivial DG algebra $(H(\mathcal{A}),0)$ by a zig-zag $\leftarrow \rightarrow\leftarrow \cdots\rightarrow $ of quasi-isomorphisms. This implies that $$\mathrm{level}_{\mathscr{D}(\mathcal{A})}^{\mathcal{A}}(k)=\mathrm{level}_{\mathscr{D}(H(\mathcal{A}))}^{H(\mathcal{A})}(k).$$
 On the other hand, we have
 $$\mathrm{level}_{\mathscr{D}(H(\mathcal{A}))}^{H(\mathcal{A})}(k)\ge \mathrm{height}(\mathrm{Ann}_{H(\mathcal{A})}k)+1=\frac{n(n-1)}{2}+1$$ by Lemma \ref{keylem}.
So we also have $\mathrm{gh.len}_{\mathcal{A}}k=\mathrm{level}_{\mathscr{D}(\mathcal{A})}^{\mathcal{A}}(k)-1=\mathrm{cl}_{\mathcal{A}}k=\mathrm{pd}_{H(\mathcal{A})}k$ if $\partial_{\mathcal{A}}\neq 0$.

In both cases, we have
\begin{align*}
\mathrm{Rouq.dim}\,\mathcal{A}&\stackrel{(c)}{=}l.\mathrm{gh.dim}(\mathcal{A})+1 \\
                              &=\sup\{\mathrm{gh.len}_{\mathcal{A}}M | M\in \mathscr{D}^c(\mathcal{A})\}+1\\
                              &\ge \mathrm{gh.len}_{\mathcal{A}}k +1 \\
                              &=\mathrm{pd}_{H(\mathcal{A})}k +1,
\end{align*}
where $(c)$ is obtained by Lemma \ref{ghlenlevel}. On the other hand,  $$\mathrm{Rouq.dim}\,\mathcal{A}\le  \mathrm{gl.dim}H(\mathcal{A})+1=\mathrm{pd}_{H(\mathcal{A})}k +1$$  by Lemma \ref{rouqdim}. Then \begin{align*}\mathrm{Rouq.dim}\,\mathcal{A} =l.\mathrm{gh.dim}(\mathcal{A})+1=\mathrm{pd}_{H(\mathcal{A})}k +1=\begin{cases}
n+1,\quad\text{if}\,\,\,\, \partial_{\mathcal{A}}=0,\\
\frac{n(n-1)}{2}+1, \quad \text{if}\,\,\, \partial_{\mathcal{A}}\neq 0.
\end{cases}
\end{align*}
\end{proof}

\section*{Acknowledgments}
 The first author is
supported by NSFC  (Grant No. 11871326) and the Innovation Program of Shanghai Municipal
Education Commission (Grant No. 12YZ031).
\def\refname{References}

\end{document}